\documentclass[12pt,reqno]{amsart}

\usepackage{amssymb, amsmath,verbatim}
\usepackage{enumerate,graphicx,epsfig}
\usepackage{color,pdfsync, subfigure,graphicx}

\newtheorem{theorem}{Theorem}[section]
\newtheorem{corollary}{Corollary}[section]
\newtheorem{proposition}{Proposition}[section]

\newtheorem{lemma}[theorem]{Lemma}
\newtheorem{claim}{Claim}[section]

\theoremstyle{definition}
\newtheorem{definition}[theorem]{Definition}

\theoremstyle{remark}
\newtheorem{remark}[theorem]{Remark}

\numberwithin{equation}{section}

\usepackage{amsmath,amsthm,amsfonts,enumerate,amssymb}
\usepackage{exscale,relsize}

\newcommand{\so}{{s_0}}

\newcommand{\gt}{\gamma_T}
\newcommand{\gr}{\gamma_R}
\newcommand{\otp}{[0,2\pi]}
\newcommand{\Rst}{\mathcal{R}^*}

\newcommand\Db{\mathrm{D}_{b}}  %Disk of radius b centered at origin

\newcommand{\abs}[1]{\lvert#1\rvert}

\newcommand{\T}[1]{T_{#1}}

\newcommand{\zero}{\boldsymbol{0}}

\newcommand{\rr}{\mathbb{R}}

\newcommand{\smo}{{\setminus \{\boldsymbol{0}}\}}

\newcommand{\tZ}{\tilde{Z}}

\newcommand{\pmpM}{{(\phi_{\min}(L),\phi_{\max}(L))}}

\newcommand{\inv}{^{-1}}

\newcommand{\dx}{\mathbf{dx}}

\newcommand{\om}{\omega}

\newcommand{\rtwo}{{\mathbb R}^2}

\newcommand{\tred}[1]{{\color{red}{#1}}}

\newcommand{\tc}[1]{\begin{quotation}\textbf{Todd's comment:\
}{\textit{#1}}\end{quotation}}
\newcommand{\vc}[1]{\begin{quotation}\textbf{Venky's comment:\
}{\textit{#1}}\end{quotation}}

\newcommand{\pf}{\begin{proof}}
\newcommand{\epf}{\end{proof}}

\newcommand{\st}{\hskip 0.3mm\big | \hskip 0.3mm}
\renewcommand{\th}{\theta}

\newcommand{\be}{\begin{equation}}
\newcommand{\ee}{\end{equation}}
\newcommand{\bea}{\begin{eqnarray}}
\newcommand{\eea}{\end{eqnarray}}
\newcommand{\bean}{\begin{eqnarray*}}
\newcommand{\eean}{\end{eqnarray*}}

\newcommand{\bel}[1]{\begin{equation}\label{#1}}

\newcommand{\eel}[1]{{\label{#1}\end{equation}}}

\newcommand{\supp}{\operatorname{supp}}

\newcommand{\cC}{{\mathcal{C}}}
\newcommand{\cD}{{\mathcal{D}}}
\newcommand{\cE}{{\mathcal{E}}}

\newcommand{\cR}{{\mathcal{R}}}

\newcommand{\shrs}{\sinh^2\rho}
\newcommand{\chrs}{\cosh^2\rho}

\newcommand{\shr}{\sinh\rho}
\newcommand{\chr}{\cosh\rho}

\newcommand{\Cc}{\mathcal{C}}
\newcommand{\Dc}{\mathcal{D}}
\newcommand{\Ec}{\mathcal{E}}

\newcommand{\Rc}{\mathcal{R}}

\newcommand{\Beq}{\begin{equation}}
\newcommand{\Eeq}{\end{equation}}
\newcommand{\beq}{\begin{equation*}}
\newcommand{\eeq}{\end{equation*}}
\newcommand{\bal}{\begin{align}}
\newcommand{\eal}{\end{align}}

\newcommand{\g}{\gamma}

\newcommand{\bp}{\begin{prob}}
\newcommand{\ep}{\end{prob}}
\newcommand{\bpr}{\begin{proof}}
\newcommand{\epr}{\end{proof}}
\newcommand{\md}{\mathrm{d}}

\author[G. Ambartsoumian]{Gaik Ambartsoumian}
 \address{Department of Mathematics, University of Texas, Arlington,
 TX, USA}
 \email{gambarts@uta.edu}

\author[V.P. Krishnan]{Venkateswaran P.\ Krishnan}
\address{Tata Institute of Fundamental Research Centre for Applicable Mathematics, Bangalore, India}
\email{vkrishnan@math.tifrbng.res.in}

\author[E.T. Quinto]{Eric Todd Quinto}
\address{Department of Mathematics, Tufts University, Medford,
 MA 02155, USA}
 \email{todd.quinto@tufts.edu}



\begin{document}
\begin{abstract}
  In this article, we consider a generalized Radon transform that comes up
  in ultrasound reflection tomography. In our model, the ultrasound emitter
  and receiver move at a constant distance apart along a circle.  We
  analyze the microlocal properties of the transform $\Rc$ that arises from
  this model. As a consequence, we show that for distributions with support
  sufficiently inside the circle, $\Rc^*\Rc$ is an elliptic
  pseudodifferential operator of order $-1$ and hence all the
    singularities of such distributions can be recovered.
\end{abstract}

%the unique recovery of a function from this transform with the
%assumption of smallness of its support.

\title[Microlocal Analysis of an Ultrasound Transform]{Microlocal
Analysis of an Ultrasound Transform with Circular Source and Receiver
Trajectories}

\maketitle

\section{Introduction}\label{sect:intro}

Ultrasound reflection tomography (URT) is one of the safest and most
cost effective modern medical imaging modalities (e.g. see
\cite{elliptic-ultrasound, elliptic-ultrasound-2, ref:Norton2D,
ref:Norton3D} and the references there). During its scanning process,
acoustic waves emitted from a source reflect from inhomogeneities
inside the body, and their echoes are measured by a receiver. This
measured data is then used to recover the unknown ultrasonic
reflectivity function, which is used to generate cross-sectional
images of the body.

%holds well for example

In a typical setup of ultrasound tomography, the emitter and receiver
are combined into one device (transducer). The transducer emits a
short acoustic pulse into the medium, and then switches to receiving
mode, recording echoes as a function of time. Assuming that the medium
is weakly reflecting (i.e. neglecting multiple reflections), and that
the speed of sound propagation $c$ is constant\footnote{This
assumption is reasonable  in ultrasound mammography, since the
speed of sound is almost constant in soft tissue.}, the echoes
measured at time $t$ uniquely determine the integrals of the
reflectivity function over concentric spheres centered at the
transducer location and radii $r=ct/2$ (see Fig.\,\ref{fig1}\,(a)
below, \cite{ref:Norton3D} and the references there). By focusing the
transducer one can consider echoes coming only from a certain plane,
hence measuring the integrals of the reflectivity function in that
plane along circles centered at the transducer location
\cite{ref:Norton2D}. Moving the transducer along a curve on the edge
of the body, and repeating the measurements one obtains a
two-dimensional family of integrals of the unknown function along
circles. Hence the problem of image reconstruction in URT can be
mathematically reduced to the problem of inverting a circular Radon
transform, which integrates an unknown function of two variables along
a two-dimensional family of circles.

In the case when the emitter and receiver are separated, the echoes
recorded by a transducer correspond to the integrals of the
reflectivity function along confocal ellipses with foci corresponding
to the locations of the emitter and receiver (see
Fig.\,\ref{fig1}\,(b)). While this more general setup has been
gaining popularity in recent years (e.g. see
\cite{elliptic-ultrasound, elliptic-ultrasound-2}), the mathematical
theory related to elliptical Radon transforms is relatively
undeveloped. In this paper we study the microlocal properties of this
transform which integrates an unknown function along a family of
ellipses.

\begin{figure*}
\centering
\begin{tabular}{cc}
\subfigure[]{\includegraphics[height=1.6in]{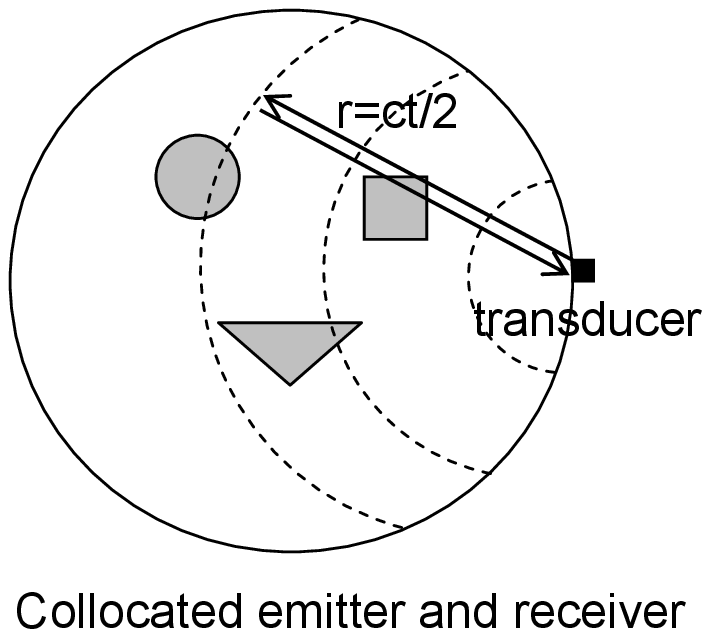}}
\subfigure[]{\includegraphics[height=1.6in]{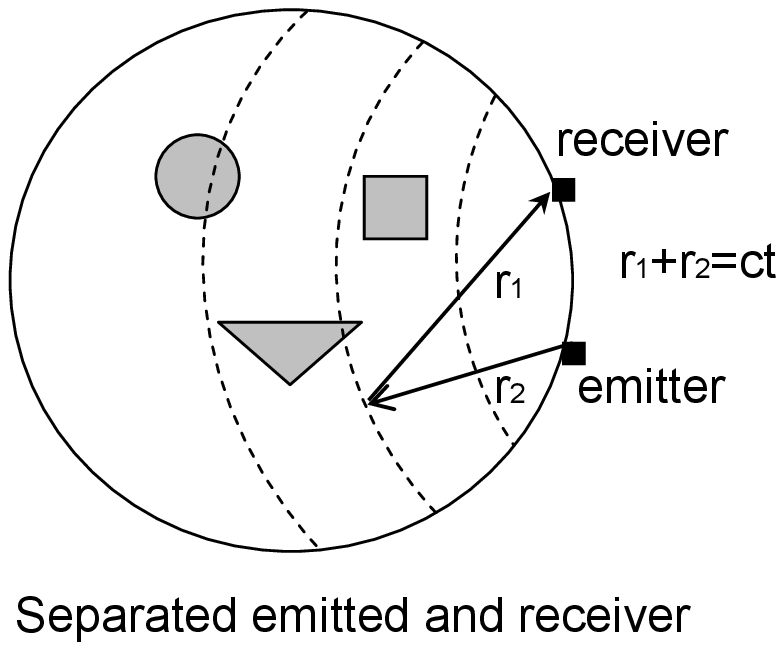}}
\end{tabular}
\caption{A sketch of integrating curves in URT} \label{fig1}
\end{figure*}

\section{Definitions and Preliminaries}\label{sect:definitions}

We will first define the elliptical Radon transform we consider,
provide the general framework for the microlocal analysis of this
transform, and show that our transform fits within this framework.
\def\mapsw#1{ \llap{$\myvcenter{\hbox{$\scriptstyle#1$}}$}
\mathlarger{\swarrow}}
\def\mapse#1{\mathlarger{\mathbf{\searrow}}\rlap{$\myvcenter{\hbox{$\scriptstyle#1$}}$}}
\newcommand{\myvcenter}[1]{\ensuremath{\raisebox{6pt}{\hbox{#1}}}}
\subsection{The Elliptical Transform}\label{sect:ell} In URT the
locations of the emitter and receiver are limited to a curve
surrounding the support of the function to be recovered and the data
taken can be modeled as integrals of the reflectivity function of the
object over ellipses with foci being the transmitter and the receiver.
In this paper we consider this curve to be a circle, which is both the
simplest case mathematically and the one most often used in practice.
By using a dilation, we can assume the circle has radius $r=1$. We
also make the restriction that the source and detector rotate around
the circle a fixed distance apart.  We denote the fixed difference
between the polar angles of emitter and receiver by $2\alpha$, where
$\alpha\in (0,\pi/2)$ (see Fig. \ref{fig-domain}) and define
\bel{def:ab}a=\sin\alpha,\qquad b=\cos\alpha .\ee

As we will see later our main result relies on the assumption that the support of the function is small enough. More precisely, we will assume our function is supported in the ball
\[\Db=\{x\in \rtwo\st |x|<b\}.\]

We parameterize the trajectories of the emitter and receiver, respectively, as 
\[
\begin{aligned}\gt(s) &=
(\cos(s-\alpha),\sin(s-\alpha))\\
 \gr(s) &=
(\cos(s+\alpha),\sin(s+\alpha))\quad \text{for $s\in
[0,2\pi].$}\end{aligned}
\]

Thus, the emitter and receiver rotate around the unit circle and
are always $2a$ units apart.  For $s\in \otp$ and $L>2a$, let
%\bel{def:E}
\[
E(s,L) = \{x\in \rtwo\st |x-\gt(s)|+|x-\gr(s)|=L\}.
\]
 Note
that the center of the ellipse $E(s,L)$ is $(b\cos s,b\sin s )$ and
$L$ is the diameter of the major axis of $E(s,L)$, the so called
\emph{major diameter}.  This is why we require $L$ to be greater than
the distance between the foci, $2a$.  Note that $E(0,L)=E(2\pi,L)$ and
so $\otp$ can be viewed as the unit circle.

\begin{figure*}
\centering
\includegraphics[height=1.6in] {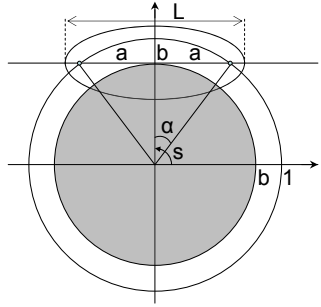}
\caption{A sketch of the domain and the notations} \label{fig-domain}
\end{figure*}

Let %\bel{def:Y} 
\[
Y=\{(s,L)\st s\in \otp,\ L>2a\},
\]
then $Y$ is the
set of parameters for the ellipses.

\begin{definition}\label{def:Rf}
Let $(s,L)\in Y$.  The elliptical Radon transform of a locally
integrable function $f:\rtwo\to\rr$ is defined as
\[
\mathcal{R}f(s,L)=\int\limits_{x\in E(s,L)} f(x) \md l(x)
\]
where $\md l$ is the arc length measure on the ellipse
$E(s,L)$.
\end{definition}
The backprojection transform is defined for $g\in C_c(Y)$ as
%\bel{def:Rstar} 
\[
\Rst g(x) = \int_{s\in \otp}
g(s,|x-\gr(s)|+|x-\gt(s)|)\md s
\]
 where the positive smooth measure $\md s$ is chosen so that $\Rst$ is
the $L^2$ adjoint of $\Rc$.  Note that $\Rst$ integrates over a
compact set, $\otp$ (or the circle), and so $\Rst$ can be composed
with $\Rc$.  Using the parametrization of ellipses $(s,L)$ one sees
that $\Rst g(x)$ integrates (with a smooth measure) over the set of
all ellipses passing through $x$.

We now briefly discuss the general framework of Guillemin and
Sternberg into which our elliptical Radon transform falls. We use this
to understand the microlocal analysis of $\Rc$.  We begin with some
general notation we will use when we discuss microlocal analysis.

\subsection{Microlocal Definitions}\label{sect:micro-def}
Let $X$ and $Y$ be smooth manifolds and let \[\cC
\subset T^*(Y)\times T^*(X), \] then we let
 \[\cC' = \{(y,\eta,x,\xi)\st (y,\eta,x,-\xi)\in \cC\}.\]
 The transpose relation is $\cC^t
\subset T^*(X)\times T^*(Y)$:
\[\cC^t = \{(x,\xi,y,\eta)\st (y,\eta,x,\xi)\in \cC\}\]

If $\cD\subset T^*(X)\times
T^*(Y)$, then the composition $\cD\circ\cC$ is  defined
\[\begin{aligned}\cD\circ \cC &= \{(x',\xi',x,\xi)
\st \exists (y,\eta)\in T^*(Y)\\
&\qquad\text{with }
(x',\xi',y,\eta)\in \cD,\ (y,\eta,x,\xi)\in \cC\}.\end{aligned}\]

\subsection{The Radon Transform and Double
Fibrations}\label{sect:micro-RT} Guillemin first put the Radon
transform into a microlocal framework, and we will use this approach
to prove Theorem \ref{thm:bolker}.  These results were first given in
the technical report \cite{Gu1975} (some of which appeared in
\cite{Gu1985}), then outlined in \cite[pp.\ 336-337, 364-365]{GS1977}.
The dependence on the measures and details of the proofs for the case
of equal dimensions were given in \cite{Q1980}.  Guillemin used the
ideas of push-forwards and pullbacks to define Radon transforms in
\cite{Gu1975} and he used these ideas to define Fourier integral
operators (FIOs) \cite{Gu1975, GS1977}.  Finally Guillemin summarized
this material in \cite{Gu1985}.

Given smooth connected manifolds $X$ and $Y$ of the same dimension, let
$Z\subset Y\times X$ be a smooth connected submanifold of codimension
$k<\dim(X)$.  We assume that the natural projections
\bel{doublefibration-genl} \def\mapsw#1{
\llap{$\myvcenter{\hbox{$\scriptstyle#1$}}$} \mathlarger{\swarrow}}
\def\mapse#1{\mathlarger{\mathbf{\searrow}}\rlap{$\myvcenter{\hbox{$\scriptstyle#1$}}$}}
{\begin{matrix}& Z&\\
&\mathlarger{\mapsw{\mathsmaller{\pi_L}}
\ \qquad\mapse{\mathsmaller{\pi_R}}}&\\
{Y} &&{X}&\end{matrix}} \ee are both fiber maps.
In this case, we call \eqref{doublefibration-genl} a \emph{double
fibration.}

Following Guillemin and Sternberg, we assume that $\pi_R$ is a
\textit{proper map;} that is, the fibers of $\pi_R:Z\to X$ are
compact.  The double fibration allows us to define a Radon transform
as follows.  For each $y\in Y$ let \[E(y) =
\pi_R\left(\pi_L\inv(\{y\})\right),\] then the sets $E(y)$ are all
diffeomorphic to the fiber of the fibration $\pi_L:Z\to Y$.  For each
$x\in X$ let \[F(x) = \pi_L\left(\pi_R\inv(\{x\})\right),\] and the
sets $F(x)$ are all diffeomorphic to the fiber of $\pi_R:Z\to X$.
Since $\pi_R$ is proper, the sets $F(x)$ are compact.  By choosing
smooth nowhere zero measures on $Z$, $X$, and on $Y$, one can then
define a smooth nowhere zero measure $\mu_y$ on $E(y)$ and $\mu_x$ on
$F(x)$ and a Radon transform
\[R f(y) = \int _{x\in E(y)} f(x) \md\mu_y(x)\]
and the dual transform is
\[R^* g(x) = \int_{y\in F(x)} g(y) \md\mu_x(y)\]
\cite{GS1977} (see also \cite[p.\ 333]{Q1980}).  Since the sets $F(x)$
are compact, one can compose $R^*$ and $R$ for $f\in C_c(X)$.  If
$\pi_R$ is not a proper map, then one needs cutoff functions to
compose $R^*$ with $R$.  We assume $E(y_1)=E(y_2)$ if and only if
$y_1=y_2$ and similarly for $F$.

Guillemin showed (\cite{Gu1975, Gu1985} and with Sternberg
\cite{GS1977}) that  $R$ is a Fourier integral distribution
associated with integration over $Z$
and canonical relation $\cC = (N^*(Z)\smo)'$.
\begin{comment}
%canonical relation \[\cC = (N^*(Z)\smo)'\,.\]
\end{comment}
To understand the properties of $R^*R$, one must investigate the
mapping properties of $\cC$.  Let $\Pi_L:\cC\to T^*(Y)$ and
$\Pi_R:\cC\to T^*(X)$ be the natural projections.  Then we have the
following diagram: \bel{micro-diagram-genl} \def\mapsw#1{
\llap{$\myvcenter{\hbox{$\scriptstyle#1$}}$} \mathlarger{\swarrow}}
\def\mapse#1{\mathlarger{\mathbf{\searrow}}\rlap{$\myvcenter{\hbox{$\scriptstyle#1$}}$}}
{\begin{matrix}& \cC&\\
&\mathlarger{\mapsw{\mathsmaller{\Pi_L}}
\ \qquad\mapse{\mathsmaller{\Pi_R}}}&\\
{T^*(Y)} &&{T^*(X)}&\end{matrix}} \ee This diagram is the microlocal
version of \eqref{doublefibration-genl}.

\begin{definition}[{\cite{Gu1975, Gu1985}}]\label{def:Bolker} Let
$X$ and $Y$ be manifolds with $\dim(Y)= \dim(X)$ and let
$\cC\subset (T^*(Y)\times T^*(X))\smo$ be a canonical relation.
Then, $\cC$ \emph{satisfies the Bolker Assumption} if \[\Pi_Y:\cC\to
T^*(Y)\] is an injective immersion. \end{definition}

This definition was originally proposed by Guillemin
\cite{Gu1975},\cite[p. 152]{Gu1985}, \cite[p.\ 364-365]{GS1977}
because a similar assumption for the Radon transform on finite sets
implies $R^*R$ is injective in this case.  Guillemin proved
that if the measures that define the Radon transform are smooth and
nowhere zero, and if the Bolker Assumption holds (and $R$ is
defined by a double fibration for which $\pi_R$ is proper), then
$R^*R$ is an elliptic pseudodifferential operator.

In the definition in \cite{Gu1975, Q1980}, the dimensions of $X$ and
$Y$ are equal, but in \cite{Gu1985}, $\dim(Y)\geq \dim(X)$.  We use
the former definition since, in our case, $\dim(Y) = \dim(X)$.  
Since we assume $\dim(Y)=\dim(X)$, if $\Pi_Y:\cC\to T^*(Y)$ is an
injective immersion, then $\Pi_Y$ maps to $T^*(Y)\smo$ and $\Pi_X$ is
also an immersion \cite{Ho1971}.  Therefore, $\Pi_X$ maps to
$T^*(X)\smo$. So, under the Bolker Assumption, $\cC\subset
(T^*(Y)\smo) \times (T^*(X)\smo)$ and so $R$ is a Fourier integral
operator according to the definition in \cite{Tre}.

\begin{comment}
\tred{\vc{Should we have ``(and another surjectivity assumption
holds)'' in the sentence above? My understanding is that the
surjectivity in question follows as a consequence of double fibration
on the base space together with the non-vanishing of measures. Am I
missing something here? }}

\tc{Is the change above in red OK.  That's what I meant, but you are
right, making it explicit that we assume $R$ is defined by a double
fibration for which $\pi_R$ is proper is good.}
\end{comment}

% \tc{\tred{Do I remember right, Venky, that you saw the Bolker
% assumption being $\Pi_L$ is a diffeomorphism, or was there anything
% about the surjectivity of $\Pi_R$ in later definitions you found of
% Bolker? I ask because, at least theoretically, without some other
% assumption, such as $\Pi_R$ is surjective to $T^*(X)\smo$ I'm not sure
% we can say $\cC^t\circ \cC = \Delta$.  As I think more about it, I
% think it must follow from the fact that $\pi_R$ is proper, so the
% fibers of $\pi_R$ are compact, which is true in our case (something
% like $\Pi_R$ of the sphere bundle of the fiber has to be a compact set
% of the sphere bundle of $T^*_x(\Db)$ (when you normalize so the
% projection maps to the sphere bundle) and the projection must be a
% local diffeomorphism).  Very interesting! What do you think?} }

For our transform $\Rc$, the
\emph{incidence relation} is \bel{def:Z} Z=\{(s,L,x)\subset Y\times
\Db \st x\in E(s,L)\}. \ee The double fibration is
\bel{doublefibration-R} \def\mapsw#1{
\llap{$\myvcenter{\hbox{$\scriptstyle#1$}}$} \mathlarger{\swarrow}}
\def\mapse#1{\mathlarger{\mathbf{\searrow}}\rlap{$\myvcenter{\hbox{$\scriptstyle#1$}}$}}
{\begin{matrix}& Z&\\
&\mathlarger{\mapsw{\mathsmaller{\pi_L}}
\ \qquad\mapse{\mathsmaller{\pi_R}}}&\\
{Y} &&{\Db}&\end{matrix}} \ee and both projections are fiber maps.
Note that the fibers of $\pi_L$ are ellipses,
$E(s,L)=\pi_R(\pi_L\inv(\{(s,L)\})$.  Furthermore, the fibers of
$\pi_R$ are diffeomorphic to $\otp$ or the circle, so $\pi_R$ is
proper.  Therefore $\Rc$ and $\Rst$ satisfy the conditions outlined
above so that Guillemin and Sternberg's framework can be applied.

\section{The Main Result}
We now state the main result of this article.

\begin{theorem}\label{thm:bolker}
  Let $\alpha>0$ be a constant and let $\g_{T}(s)=( \cos (s-\alpha),
  \sin(s-\alpha))$ and $\g_{R}(s)=( \cos(s+\alpha),\sin(s+\alpha))$ for
  $s\in \otp$ be the trajectories of the emitter and receiver
  respectively. Denote by $\Ec'(\mathrm{D}_{b})$ the space of compactly
  supported distributions supported in the disc, $\Db$, of radius $b$
  centered at $0$, where $b=\cos \alpha$. The elliptical Radon transform
  $\Rc$ when restricted to the domain $\Ec'(\mathrm{D}_{b})$ satisfies the
  microlocal conditions in Section \ref{sect:micro-RT}, and $\Rc$ is an
  elliptic Fourier integral operator (FIO) of order $-1/2$.  Let
  $\cC\subset T^*(Y)\times T^*(\Db)$ be the canonical relation associated
  to $\Rc$.  Then, $\cC$ satisfies the Bolker Assumption (Definition
  \ref{def:Bolker}).
\end{theorem}

As a consequence of this result, we have the following corollary.

\begin{corollary}\label{cor:R*R}
  The composition of $\Rc$ with its $L^{2}$ adjoint $\Rc^{*}$ when
restricted as a transformation from $\cE'(\mathrm{D}_{b})$ to
$\cD'(\mathrm{D}_{b})$ is an elliptic pseudo-differential operator of order
$-1$.
\end{corollary}

\begin{remark}\label{remark:R*R-sing} This corollary shows that, for
$\supp f\subset \Db$, the singularities of $\Rc^*\Rc f$ are at the
same locations and co-directions as the singularities of $f$, that is,
the wavefront sets are the same.  In other words, $\Rst \Rc$
reconstructs all the singularities of $f$.  If $D$ is an elliptic
differential operator on $Y$, then one can create an elliptic local
reconstruction operator $\Rc^* D\Rc$ that will emphasize
singularities.  REU research student Howard Levinson
\cite{Howiethesis} refined and implemented an algorithm of Prof.\
Quinto's of this form and the algorithm does show all the
singularities of $f$.  Because of the derivative, the algorithm
highlights boundaries, and this type of algorithm generalizes Lambda
Tomography \cite{FRS, FFRS}.\end{remark}

\section{Proofs of Theorem \ref{thm:bolker} and Corollary
\ref{cor:R*R}}\label{sect:proofs}

\begin{proof}[Proof of Theorem \ref{thm:bolker}]

First, we will calculate $\cC = (N^*(Z)\smo)'$ where $Z$ is given by
\eqref{def:Z}.  Then most of the proof is devoted to showing that
$\cC$ satisfies the Bolker Assumption.

  Since $Z$ is defined by the equation
\[L-|x-\gr(s)|-|x-\gt(s)|=0,\]     the differential of the function is
a basis for $N^*(Z)$.  That differential is
%\bel{dZ}
\[
\begin{aligned}
% &\left(\frac{(x-\gr(s))\cdot\gr'(s)}{|x-\gr(s)|}
% +\frac{(x-\gt(s))\cdot \gt'(s)}{|x-\gt(s)|}\right)ds
% \\&\quad -dL
% -\left(\frac{x-\gr(s)}{|x-\gr(s)|}
% +\frac{x-\gt(s)}{|x-\gt(s)|}\right)\dx\\
&= -\left(\frac{x\cdot\gr'(s)}{|x-\gr(s)|} +\frac{x\cdot
\gt'(s)}{|x-\gt(s)|}\right)\md s +\md L \\
&\qquad -\left(\frac{x-\gr(s)}{|x-\gr(s)|}
+\frac{x-\gt(s)}{|x-\gt(s)|}\right)\md x.\end{aligned}
\]
Note that we
are using the fact that $\gr(s)$ and $\gr'(s)$ are orthogonal (as are
$\gt(s)$ and $\gt'(s)$). Therefore, $\cC=\left(N^*(Z)\smo\right)'$ is
\bel{def:C}\begin{aligned} \cC
=\Bigg{\{}\Bigg(s,L,\omega\left(-\left(\frac{x\cdot\gr'(s)}{|x-\gr(s)|} +\frac{x\cdot
\gt'(s)}{|x-\gt(s)|}\right)\md s +\md L\right),\\
\qquad\qquad x,
\omega\left(\frac{x-\gr(s)}{|x-\gr(s)|}
+\frac{x-\gt(s)}{|x-\gt(s)|}\right)\md x\Bigg)\\
\qquad\qquad\qquad\qquad
\st (s,L,x)\in Z,\ \om\in \rr\setminus \{0\}\Bigg{\}}.\end{aligned}\ee
The Schwartz kernel of $\Rc$ is integration
on $Z$ (e.g., \cite[Proposition 1.1]{Q1980})
and so is a Fourier integral distribution associated to $\cC$
\cite{Gu1985}.

Now we show that the projection $\Pi_L$ satisfies the Bolker
Assumption.  If $\lambda \in \cC$ the projection is given by
\eqref{def:C} %\bel{PiL:genl}
\[
\Pi_L(\lambda ) =
\left(s,L,\omega\left(-\left(\frac{x\cdot\gr'(s)}{|x-\gr(s)|}
+\frac{x\cdot \gt'(s)}{|x-\gt(s)|}\right)\md s +\md L\right)\right).
\]
From
this, we have determined, $s$, $L$, and $\om$ and we need to find $x$
knowing that $x\in E(s,L)$.

The easiest way to do this is to develop coordinates, first on the
ellipses, then on $Z$, and finally on $\cC$.

We will reduce to the case $s=\pi/2$, so first we give
 coordinates on the ellipse $E(\pi/2,L)$ (see Figure \ref{ellpt-coord-pic}), which has foci $(\pm a,b)$:
for $\phi\in \otp$ and $L>2a$ we define
\bel{x'-coords}
\begin{gathered} x'=x'(L,\phi)  =
(a\cosh\rho\cos \phi,b+a\sinh\rho\sin\phi)\\
\text{where} \ \ \rho = \arccos(2L/a).\end{gathered}
\ee

\begin{figure*}
\centering
\includegraphics[height=3in]{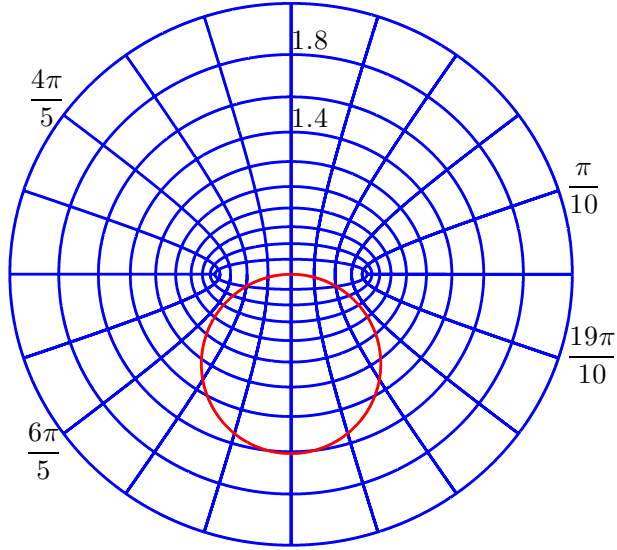}
%\vspace{-5mm}
\caption{Elliptic Coordinate System. The radius of the circle in the figure is $b$.} \label{ellpt-coord-pic}
\end{figure*}
Note that the ellipse $E(s,L)$ meets the ball $\Db$ if and only if
$L\in (2a,2\sqrt{a^2+4b^2})$, and for each $L\in (2a,2\sqrt{a^2+4b^2})$,
there is an interval $(\phi_{\min}(L),\phi_{\max}(L))$ with
$\pi<\phi_{\min}(L)<\phi_{\max}(L)<2\pi$ such that $x'(L,\phi)\in
\Db$ if and only if $\phi\in (\phi_{\min}(L),\phi_{\max}(L))$.

Next we let $\T{\th}$ be the counterclockwise rotation about the
origin through $\th$ radians.  For $(s,L)\in Y$, this gives
coordinates on the ellipse $E(s,L)$ %\bel{x-coords}
\[
x(s,L,\phi) =
\T{s-\pi/2}\left(x'(L,\phi)\right)
\]
since $\T{s-\pi/2}$ rotates
$\gr(\pi/2)$ to $\gr(s)$ and rotates $\gt(\pi/2)$ to $\gt(s)$.  This
rotation preserves distances and so it does not change the major
diameter of the ellipse, $L$.  Furthermore, for $\phi\in
(\phi_{\min}(L),\phi_{\max}(L))$, $x(s,L,\phi)\in \Db$.

We get coordinates on $Z$ as follows.  Let %\bel{def:tZ}
\[
\begin{aligned}\tZ=
\{(s,L,\phi)\st s\in \otp,\ L\in (2a,2\sqrt{a^2+4b^2}),\\
\phi\in (\phi_{\min}(L),\phi_{\max}(L))\}.\end{aligned}
\]
This provides
local coordinates above $\Db$ on $Z$: \[ \tZ\ni(s,L,\phi)\mapsto
z(s,L,\phi) = \left(s,L,x(s,L,\phi)\right).\] Finally, this gives
coordinates on $\cC$: %\bel{coords:C}
\[
\begin{aligned}
\lambda(s,L,\phi,\om) &=
(s,L,\omega\left(-\left(\frac{x\cdot\gr'(s)}{|x-\gr(s)|} +\frac{x\cdot
\gt'(s)}{|x-\gt(s)|}\right)\md s +\md L\right),\\
&\qquad\qquad x,
\omega\left(\frac{x-\gr(s)}{|x-\gr(s)|}
+\frac{x-\gt(s)}{|x-\gt(s)|}\right)\md x\big)\\
&\text{where} \ \ (s,L,\phi,\om)\in \tZ\times (\rr\setminus \{0\}),\\
&\text{and}\ x=x(s,L,\phi)\end{aligned}
\] and this gives $\Pi_L$ in
coordinates as %\bel{coords:PiL}
\[
\begin{aligned}\Pi_L(\lambda(s,L,\phi,\om))& =
\Big(s,L,-\omega\left(\frac{x\cdot\gr'(s)}{|x-\gr(s)|} +\frac{x\cdot
\gt'(s)}{|x-\gt(s)|}\right)\md s\\ &\qquad+\om \md L\Big).  \end{aligned}
\]

Using a series of estimates, we will show the following
claim.

\begin{claim}\label{lemma:deriv}
For $(s,L,\phi)\in \tZ$, the derivative of
\bel{genl:ds}\frac{x(s,L,\phi)\cdot\gr'(s)}{|x(s,L,\phi)-\gr(s)|}
+\frac{x(s,L,\phi)\cdot \gt'(s)}{|x(s,L,\phi)-\gt(s)|}\ee with
respect to $\phi$ is never zero.\end{claim}

This claim shows that $\Pi_L$ is injective for the following reasons.
Since $s,L$ and $\om\neq 0$ are given from the projection $\Pi_L$, to
show $\Pi_L$ is injective, all we need to show is that for fixed
$(s,L)$ with $s\in \otp$ and $L\in (2a,2\sqrt{a^2+4b^2}),$
\eqref{genl:ds} determines $\phi\in \pmpM$ uniquely.  However, Claim
\ref{lemma:deriv} shows for fixed $(s,L)$ and $\phi$ in the interval
$\pmpM$ that \eqref{genl:ds} is either strictly increasing or strictly
decreasing.  Since $\pmpM$ is an interval, this will show that
\eqref{genl:ds} is an injective function of $\phi$ and therefore
$\Pi_L$ is injective.

Next, we will use Claim \ref{lemma:deriv} to argue that $\Pi_L$ is an
immersion.  Since $s,L$ and $\om\neq 0$ are given directly from the projection
$\Pi_L$, writing out the derivative matrix of $\Pi_L$ shows that one
only needs to prove for $(s,L,\phi)\in \tZ$ and $\om\neq 0$ that the
derivative of \eqref{genl:ds} with respect to $\phi$ is nowhere zero.
This again follows from the claim.

\begin{proof}[Proof of Claim \ref{lemma:deriv}]
We  prove this by making a reduction to $s=\pi/2$ and then  by a
series of estimates.

Using rotation invariance: we apply $\T{\pi/2-s}$ on $\rtwo$, and use
the facts that distances and dot products are preserved
(so e.g.,
$x\cdot\gr'(s) = x'\cdot\gr'(\pi/2)$ and
$|x-\gr(s)|=|x'-\gr(\pi/2)|$) where
\[\begin{gathered}
\T{\pi/2-s}(x(s,L,\phi)) = x'(L,\phi),\\
\T{\pi/2-s}(\gr(s))=\gr(\pi/2)=(-a,b),\\
\T{\pi/2-s}(\gr'(s))=\gr'(\pi/2)=(-b,-a))\\
\T{\pi/2-s}(\gt(s))=\gt(\pi/2)=(a,b),\\
\T{\pi/2-s}(\gt'(s))=\gt'(\pi/2)=(-b,a)\end{gathered}\] This gives the
simplified expression that is equal to \eqref{genl:ds}:
%\bel{pi/2:ds}
\[
S(L,\phi)=\left(\frac{x'(L,\phi)\cdot(-b,-a)}{|x'(L,\phi)-(-a,b)|}
+\frac{x'(L,\phi)\cdot (-b,a)}{|x'(L,\phi)-(a,b)|}\right).
\]

We have reduced showing $\Pi_L$ is an injective immersion to showing
%\bel{deriv-goal}
\[
\begin{gathered}\frac{\partial}{\partial
\phi}S(L,\phi)\neq 0 \\ \text{ for }L\in
(2a,2\sqrt{a^2+4b^2}),\ \phi\in  \pmpM.\end{gathered}
\]

%*****************************************************
Using (\ref{x'-coords}) we obtain
\[\begin{gathered}|x'(L,\phi)-(-a,b)|=\sqrt{a^2(\cosh\rho\cos\phi+1)^2+a^2\sinh^2\rho\sin^2\phi}\\
=a(\cosh\rho+\cos\phi),\end{gathered}\]
\[\begin{gathered}|x'(L,\phi)-(a,b)|=\sqrt{a^2(\cosh\rho\cos\phi-1)^2+a^2\sinh^2\rho\sin^2\phi}\\
=a(\cosh\rho-\cos\phi)\end{gathered}\]
and
\begin{align*}
S(L,\phi)&=\frac{-b\cosh\rho\cos\phi-b-a\sinh\rho\sin\phi}{\cosh\rho+\cos\phi}\\
&+\frac{-b\cosh\rho\cos\phi+b+a\sinh\rho\sin\phi}{\cosh\rho-\cos\phi}.
\end{align*}

Adding these two fractions, and simplifying the expression yields the following:
\begin{equation*}
S(L,\phi)=H(\rho,\phi)=\displaystyle{\frac{2\shr\,\cos\phi\;(a\sin\phi-b\shr)}{(\chr+\cos\phi)(\chr-\cos\phi)}}
\end{equation*}
Our problem therefore is reduced to showing that $\frac{\partial}{\partial\phi}H(\rho,\phi)\neq 0$.

Denote $t=\cos\phi$, then $\sin\phi=-\sqrt{1-t^2}$, since
$\phi\in[\pi,2\pi]$.

\noindent Consider a new function $\tilde{H}(t)$ defined as follows (the term $-2\shr$ does not matter):

\begin{equation*}
\tilde{H}(t)=\displaystyle{\frac{b\,\shr t+at\sqrt{1-t^2}}{\chrs-t^2}}.
\end{equation*}

\begin{align*}
\tilde{H}'(t)&=\frac{1} {\left(\chrs-t^2\right)^2}
\Big[2t\left(bt\shr+at\sqrt{1-t^2}\right)\\
&\quad+\left(b\,\shr
+a\sqrt{1-t^2}-\displaystyle\frac{at^2}{\sqrt{1-t^2}}\right)
  \left(\chrs-t^2\right)\Big].\end{align*}
Denoting the numerator of the above expression by $N$, and
distributing the product we get

\bel{numerator}\begin{aligned}
N&=\chrs\left(b\,\shr+a\sqrt{1-t^2} -
\displaystyle\frac{at^2}{\sqrt{1-t^2}}\right)\\
&\qquad +t^2\left(b\shr+a\sqrt{1-t^2}+\displaystyle\frac{at^2}{\sqrt{1-t^2}}\right)\\
 &=\left(\chrs+t^2\right)\left(b\shr+a\sqrt{1-t^2}\right)-\left(\chrs-t^2\right)\frac{at^2}{\sqrt{1-t^2}}.
\end{aligned}\ee
Denoting \[k=\shr, \ \ \text{we have} \ \ \chrs=1+k^2,\] and
$$\begin{aligned}
N&=\sqrt{1-a^2}\,k\,(1+k^2)+a(1+k^2)\sqrt{1-t^2}+\sqrt{1-a^2}\;k\,t^2\\
&\qquad+ at^2\sqrt{1-t^2}+\displaystyle\frac{at^4}{\sqrt{1-t^2}}
-\displaystyle\frac{a(1+k^2)t^2}{\sqrt{1-t^2}}
\end{aligned}$$
\noindent Notice that all terms in the above expression are non-negative except the last one. Hence, in order to show that $N>0$
it is enough to show that $N_1>0$, where
$$
N_1=a(1+k^2)\sqrt{1-t^2}+at^2\sqrt{1-t^2}+\displaystyle\frac{at^4}{\sqrt{1-t^2}}-\displaystyle\frac{a(1+k^2)t^2}{\sqrt{1-t^2}}.
$$
Since multiplication by $\sqrt{1-t^2}$ does not change the algebraic sign, $N>0$ if and only if $\sqrt{1-t^2}\;N_1>0$, which is
equivalent to
\begin{equation}\label{star7}
t^2\le \displaystyle\frac{1+k^2}{1+2k^2}.
\end{equation}

\begin{definition}
We will call the pair $(\rho,\phi)$ \textit{admissible} if the point defined by elliptic coordinates $(\rho,\phi)$ is located inside $D(0,b)$, the disc of radius $b$ centered at the origin. % denoted by $D(0,b)$.
\end{definition}

We need to show that $N>0$ for all admissible pairs $(\rho,\phi)$.
\begin{comment}
\begin{figure*}
\centering
\includegraphics[height=3in]{ellipticcoordinates_thin.eps}
\vspace{-5mm}
\caption{Elliptic Coordinate System.} \label{ellpt-coord-pic}
\end{figure*}

%%%%%%%%%%%%%%%%%%%%%%%%%%%%%%%%%%%%%%%%%%%           Lemma 1       %%%%%%%%%%%%%%%%%%%%%%%%%%%%%%%%%%%%%%
\end{comment}
\begin{lemma}\label{Admissible pairs lemma 1}
If $|\cos\phi|>b$ then $(\rho,\phi)$ is not admissible for any $\rho>0$.
\end{lemma}

\begin{proof}
Recall the coordinate system $(\rho,\phi)$ (see Fig. \ref{ellpt-coord-pic}) for $x=(x_{1},x_{2})$ defined by
\begin{equation}\label{ellpt-coord-def}
\left\{
\begin{array}{l}
x_{1}=a\chr \cos\phi\\
x_{2}=b+a\shr\sin\phi
\end{array}
\right.
\end{equation}
where $a>0$ is fixed, $\rho>0$, and $\phi\in[0,2\pi]$. It is easy to notice that the coordinate curves (hyperbolas) corresponding to fixed values of $\phi$   intersect the disc $D(0,b)$ only for limited values of $\phi$. To find the range of these values consider the following system of equations:

\begin{equation*}
\left\{
\begin{array}{rcl}
\left(\displaystyle\frac{x_{1}}{a\cos\phi}\right)^2-\left(\displaystyle\frac{x_{2}-b}{a\sin\phi}\right)^2&=&1,\;\;
\phi\ne \displaystyle\frac{k\pi}{2}, \;\; k=0,\ldots,4\\
\\
x_{1}^{2}+x_{2}^2&=&b^2.
\end{array}
\right.
\end{equation*}
Simplifying the system one gets a quadratic equation with respect to $(x_{2}-b)$
$$
(x_{2}-b)^2+2b\sin^2\phi(x_{2}-b) +a^2\,\sin^2\phi\;\cos^2\phi=0,
$$
which does not have any real roots if $|\tan\phi|<\displaystyle\frac{a}{b}$. Recalling that $b=\sqrt{1-a^2}$ we conclude that
the hyperbolas corresponding to coordinate curves with constant $\phi$ intersect the domain of the function support $D(0,b)$
only when $|t|\le b= \sqrt{1-a^2}$.
\end{proof}

\begin{comment}
%%%%%%%%%%%%%%%%%%%%%%%%%%%%%%%%%%%%%%%%%%%           Lemma 2       %%%%%%%%%%%%%%%%%%%%%%%%%%%%%%%%%%%%%%
\end{comment}
\begin{lemma}\label{Admissible pairs lemma 2}
If $\chrs>\displaystyle\frac{4-3a^2}{a^2}$ then $(\rho,\phi)$ is not admissible for any $\phi$.
\end{lemma}

\begin{proof} From equations (\ref{ellpt-coord-def}) it is easy to notice that the coordinate curves (ellipses) corresponding to fixed  values of $\rho$ do not intersect the disc $D(0,-b,b)$ if $\rho$ is large. Namely, the largest value of $\rho$ for which the corresponding ellipse intersects the disc satisfies the equation $2b=a\shr$. Since $\chrs=1+\shrs$, and $b^2=1-a^2$ we get that
$$
\chrs\le\displaystyle\frac{4-3a^2}{a^2}
$$
for each coordinate ellipse that intersects the disc.
\end{proof}

%%%%%%%%%%%%%%%%%%%%%%%%%%%%%%%%%%%%%%%%%%%       Proposition 1       %%%%%%%%%%%%%%%%%%%%%%%%%%%%%%%%%%%%%%

\begin{proposition}\label{proposition 1}
If $\alpha>0.8$, then for all admissible pairs $(\rho,\phi)$ we have $N>0$.
\end{proposition}
\begin{proof}
Recall from Lemma \ref{Admissible pairs lemma 2} that for admissible pairs $k^2+1\le \displaystyle\frac{4-3a^2}{a^2}$, hence
$0<k\le\displaystyle\frac{2\sqrt{1-a^2}}{a}$. Consider the right hand side of equation \eqref{star7}
$$
f(k)=\frac{1+k^2}{1+2k^2}.
$$
The function $f(k)$ is monotonically decreasing and for admissible pairs reaches its minimum
$$f(k_{\mathrm{min}})=\displaystyle\frac{4-3a^2}{8-7a^2},\;\;
k_{\mathrm{min}}=\displaystyle\frac{2\sqrt{1-a^2}}{a}.$$
Hence, if we show that under the conditions of the theorem
\begin{equation}\label{E1}
t^2<\displaystyle\frac{4-3a^2}{8-7a^2}
\end{equation}
then by equation (\ref{star7}) and the preceding argument $N>0$. At the same time by Lemma \ref{Admissible pairs lemma 1}
we know that $t^2<1-a^2$. So if $a>0$ satisfies
\begin{equation}\label{E2}
1-a^2<\displaystyle\frac{4-3a^2}{8-7a^2},
\end{equation}
then \eqref{E1} is satisfied. To finish the proof, notice that the hypothesis of the proposition $\alpha>0.8$ implies \eqref{E2}.
\end{proof}

\begin{comment}
%%%%%%%%%%%%%%%%%%%%%%%%%%%%%%%%%%%%%%%%%%%       Proposition 2       %%%%%%%%%%%%%%%%%%%%%%%%%%%%%%%%%%%%%%
%If $b\geq \frac{1}{2\sqrt{2}}$, then $N>0$.
\end{comment}

\begin{proposition}\label{proposition 2}
If $\alpha<1.2$, then for all admissible pairs $(\rho,\phi)$ we have $N>0$.

\end{proposition}

\begin{proof}
From Lemma \ref{Admissible pairs lemma 1}, we need to consider only those pairs $(\rho,\phi)$ for which $|t|\leq
b$, where $t=\cos\phi$.  We will show that $N$ is positive for this range
of $t$.

The term $N$ in \eqref{numerator} can be rewritten as
\Beq\label{Numerator rewritten}
N(t) =
b\sqrt{1-t^2} \shr(\chrs + t^2) + a (\chrs - t^2(\chrs +
\shrs))
\Eeq
Using the fact that $b\geq \abs{t}$, we have $\sqrt{1-t^2}\geq
\sqrt{1-b^2} = a$ and we substitute this into \eqref{Numerator rewritten} and use the fact
that $b=\sqrt{1-a^2}$ to get
\[
N(t)\geq ba\shr(\chrs +
t^2) + a (\chrs - t^2(\chrs + \shrs).
\]

Factoring $a$, we are left with wanting to show that
\[
b\shr(\chrs
+t^2)+ \chrs - t^2(\chrs + \shrs)\geq 0.
\]
Isolating $t^{2}$, we need to show that
\bel{est3}\frac{(b\shr+1)\chrs}{\chrs+\shrs - b\shr} \geq t^2.
\ee
(Note that the term $\chrs+\shrs - b\shr$ is always positive.)
We show that for $b\geq \frac{1}{2\sqrt{2}}$, the left hand side of \eqref{est3} is greater than $1$.
Solving the inequality by let $x=\shr$ we get the quadratic inequality
\[
bx^2 - x + 2b \geq 0\qquad \text{ with roots }\ x =
\frac{1\pm \sqrt{1-8b^2}}{2b}.
\] There are no roots if $1<8b^2$ or
$1/(2\sqrt{2})<b$.  Since $b=\cos \alpha$, this corresponds to $\alpha < 1.21$ radians.
\begin{comment}
% which
%is wider than $\pi/4$.
\end{comment}
\end{proof}

We have now proved Claim \ref{lemma:deriv}, and from the discussion
following this claim, we have that $\Pi_L$ is an injective immersion:
the Bolker Assumption holds.\end{proof}

Now that we know the Bolker Assumption holds, as mentioned after
Definition \ref{def:Bolker}, the projections $\Pi_L$ and $\Pi_R$ map
away from $\zero$.  Therefore, $\Rc$ is a Fourier integral operator
\cite{Tre}.  Since the measure of integration on the ellipses,
arc length, is nowhere zero in $\Db$, $\Rc$ is elliptic.
The order of $\Rc$ is given by $(\dim(Y)-\dim(Z))/2$ 
(see e.g., \cite[Theorem 1]{Gu1985} which gives the order of $\Rc^{*}\Rc$).
In our case, $Z$ has dimension $3$ and $Y$ has dimension $2$, hence $\Rc$
has order $-1/2$.

% (see e.g., \cite[Theorem 1]{Gu1985}).
\begin{comment}
%Duistermaat, Page 41, Equation 2.4.22 as well

\tc{I used Guillemin for the order of $\Rc$.  Should I explain that
Guillemin really states that $R^*R$ has order $\dim(Y)-\dim(Z)$ so $R$
has dimension that over 2?}
\end{comment}

\end{proof}

\begin{proof}[Proof of Corollary \ref{cor:R*R}] The proof follows from
  Guillemin's result \cite[Theorem 1]{Gu1985} as a consequence of
Theorem \ref{thm:bolker}. We will outline his proof that $\Rst\Rc$ is
an elliptic pseudodifferential operator since the proof for our
transform is simple and instructive. 

By Theorem \ref{thm:bolker}, $\Rc$ is an elliptic Fourier integral
operator associated with $\cC$.  By the standard calculus of FIO,
$\Rst$ is an elliptic FIO associated to $\cC^t$.  Note that we can
compose $\Rst$ and $\Rc$ because $\Rc^*$ integrates over $\otp$ (in
general, because $\pi_R$ is proper).
\begin{comment}
\tc{Is the above text OK? right? We now explain why $\cR$ is elliptic
in the proof of Theorem \ref{thm:bolker}}
\end{comment}
\begin{comment}
By Theorem \ref{thm:bolker}, $\Rc$ and $\Rst$ are defined by a double
fibration satisfying Guillemin's conditions and therefore $\Rc$ is a
Fourier integral distribution associated with integration over $Z$.
As noted after Definition \ref{def:Bolker}, this implies $\Rc$ and
$\Rc^*$ are FIOs.  \tred{Furthermore, because the measure of
integration for $\Rc$ and hence $\Rst$ are nowhere zero, these
operators are elliptic.} Note that we can compose $\Rst$ and $\Rc$
because $\Rc^*$ integrates over $\otp$.
\end{comment}

Because the Bolker Assumption holds above $\Db$, $\cC$ is a local
canonical graph and so the composition $\Rc^* \Rc$ is a FIO for
functions supported in $\Db$.  Now, because of the injectivity of
$\Pi_Y$, $\cC^t\circ \cC\subset\Delta$ where $\Delta $ is the diagonal
in $(T^*(\Db)\smo)^2$ by the clean composition of Fourier integral
operators \cite{DG}.  

To show $\cC=\Delta$, we need to show $\Pi_R:\cC\to T^*(\Db)\smo$ is
surjective.  This will follow from \eqref{def:C} and a geometric
argument.  Let $(x,\xi)\in T^*(\Db)\smo$.  We now prove there is a
$(s,L)\in Y$ such that the ellipse $E(s,L)$ is conormal to $(x,\xi)$.
First note that any ellipse $E(s,L)$ that contains $x$ must have
$L=|x-\g_R(s)|+|x-\g_T(s)|$.  As $s$ ranges from $0$ to $2\pi$ the
normal line at $x$ to the ellipse $E(s,|x-\g_R(s)|+|x-\g_T(s)|)$ at
$s$ rotates completely around $2\pi$ radians and therefore for some
value of $\so\in [0,2\pi]$ $E(\so,|x-\g_R(\so)|+|x-\g_T(\so)|)$ must
be conormal $(x,\xi)$.  Since the ellipse is given by the equation
$L=|x-\g_R(s)|+|x-\g_T(s)|$, its gradient is normal to the ellipse at
$x$; conormals co-parallel this gradient are exactly of the form $ \xi
= \omega\left(\frac{x-\gr(\so)}{|x-\gr(\so)|}
+\frac{x-\gt(\so)}{|x-\gt(\so)|}\right)\dx$ for some $\omega\neq 0$.
Using \eqref{def:C}, we see that for this $\so$, $x$, $\omega$ and
$L=|x-\g_R(\so)|+|x-\g_T(\so)|$, there is a $\lambda\in \cC$ with
$\Pi_R(\lambda) =(x,\xi)$.  This finishes the proof that $\Pi_R$ is
surjective.  Note that one can also prove this using the fact that
$\pi_R$ is a fibration (and so a submersion) and a proper map, but our
proof is elementary.  This shows that $\Rc^*\Rc$ is an elliptic
pseudodifferential operator viewed as an operator from $\Ec'(D_{b})\to
\Dc'(D_{b})$.  Because the order of $\Rst$ and $\Rc$ are $-1/2$, the
order of $\Rst\Rc$ is $-1$.
\end{proof}

\section{Acknowledgements}

All authors thank the American Mathematical Society for organizing the
Mathematical Research Communities Conference on Inverse Problems that
encouraged our research collaboration.  The first and third author
thank MSRI at Berkeley for their hospitality while they discussed
these results.  The first author was supported in part by DOD CDMRP
Synergistic Idea Award BC063989/W81XWH-07-1-0640, by Norman
Hackerman Advanced Research Program (NHARP) Consortium Grant
003656-0109-2009 and by NSF grant DMS-1109417. The second
author was supported in part by NSF Grants DMS-1028096 and DMS-1129154
(supplements to the third author's NSF Grant DMS-0908015) and
DMS-1109417. Additionally he thanks Tufts University for providing an
excellent research environment and the University of Bridgeport for
the support he received as a faculty member there. The third author
was supported in part by NSF Grant DMS-0908015.
\bibliographystyle{plain}

\end{document}